\documentclass[a4paper,11pt]{amsart}

\usepackage{amsmath,amssymb}
\usepackage{amsthm}
\usepackage{mathrsfs}
\usepackage{enumitem}
\usepackage{mathtools}
\usepackage{longtable,filecontents,ltxtable}
\usepackage[all]{xy}

\theoremstyle{plain}
\newtheorem{theorem}{Theorem}[section]
\newtheorem{proposition}[theorem]{Proposition}
\newtheorem{lemma}[theorem]{Lemma}
\newtheorem{corollary}[theorem]{Corollary}
\newtheorem{problem}[theorem]{Problem}

\theoremstyle{definition}
\newtheorem{definition}[theorem]{Definition}
\newtheorem{conjecture}[theorem]{Conjecture}
\newtheorem*{acknowledgements}{Acknowledgements}

\theoremstyle{remark}
\newtheorem{remark}[theorem]{Remark}
\newtheorem{notation}[theorem]{Notation}
\newtheorem{convention}[theorem]{Convention}

\newtheorem{condition}{Condition}

\numberwithin{equation}{theorem}

\DeclareMathOperator{\Pic}{Pic}
\DeclareMathOperator{\NE}{NE}

\DeclareMathOperator{\Psef}{Psef}
\DeclareMathOperator{\Bigcone}{Big}

\newcommand{\sE}{\mathscr{E}}
\newcommand{\sF}{\mathscr{F}}

\newcommand{\sC}{\mathscr{C}}
\newcommand{\sS}{\mathscr{S}}
\newcommand{\sN}{\mathscr{N}}

\newcommand{\cO}{\mathcal{O}}

\newcommand{\bC}{\mathbb{C}}

\newcommand{\bP}{\mathbb{P}}
\newcommand{\bQ}{\mathbb{Q}}

\newcommand{\bZ}{\mathbb{Z}}
\newcommand{\bO}{\mathbb{O}}

\newcommand{\pbP}[2]{\left(\bP^{#1}\right)^{#2}}
\newcommand{\pbQ}[2]{\left(\bQ^{#1}\right)^{#2}}
\newcommand{\pbPB}[2]{\big(\bP(#1)\big)^{#2}}

\AtBeginDocument{%
   \def\MR#1{}
}

\setenumerate{label=(\arabic*),
font=\normalfont
}

\title{Extremal rays and nefness of tangent bundles}
\author[A. KANEMITSU]{Akihiro KANEMITSU}
\date{\today}
\address{Graduate School of Mathematical Sciences\\The University of Tokyo\\3-8-1 Komaba\\Meguro-ku, Tokyo 153-8914, Japan}
\email{kanemitu@ms.u-tokyo.ac.jp}

\begin{document}

\begin{abstract}
In view of Mori theory, rational homogenous manifolds satisfy a recursive condition: every elementary contraction is a rational homogeneous fibration and the image of any elementary contraction also satisfies the same property.
In this paper, we show that a smooth Fano $n$-fold with the same condition and Picard number greater than $n-6$ is either a rational homogeneous manifold or the product of $n-7$ copies of $\mathbb{P}^1$ and a Fano $7$-fold $X_0$ constructed by G. Ottaviani.
We also clarify that $X_0$ has non-nef tangent bundle and in particular is not rational homogeneous.
\end{abstract}

\maketitle

\section*{Introduction}
In view of Mori theory, extremal rays or their contractions play an important role to study projective manifolds with non-nef canonical bundle, for example, Fano manifolds.
In this viewpoint, we notice that any rational homogenous manifold satisfies the following recursive condition:

\begin{condition}\label{cond:*}
For every sequence of elementary contractions
\[
X  \xrightarrow{f_1}  X_1  \xrightarrow{f_2} \cdots \xrightarrow{f_{m-1}} X_{m-1} \xrightarrow{f_m}  X_m,
\]
each $f_i$ is a rational homogeneous fibration.
Here a contraction is called a \emph{rational homogeneous fibration} if it is smooth and every fiber is a rational homogeneous manifold.
\end{condition}

In this paper, we study Fano manifolds with Condition~\ref{cond:*}.
The above condition is motivated by the following conjecture due to Campana and Peternell, which is a generalization of Mori's result  \cite{Mor79} and  known to be true for $n$-folds with Picard number $\rho_X > n-5$ \cite{CP91,CP93,Hwa06,Mok02,Wat14a,Wat15,Kan15a,Kan15b}:

\begin{conjecture}[Campana-Peternell conjecture \cite{CP91}]
\label{conj:CP}
Every Fano manifold $X$ with nef tangent bundle is a rational homogeneous manifold.
\end{conjecture}
Indeed, by reviewing the proof in above cited references \cite{CP91,CP93,Wat14a,Wat15,Kan15b} on Conjecture~\ref{conj:CP} for manifolds with Picard number greater than one, it turns out that Fano $n$-folds with Condition~\ref{cond:*} and Picard number $\rho_X>n-5$ are rational homogeneous manifolds by a similar argument.
We shall give a sketch of this in the present paper (see Sect.~3).

For this fact, one might hope that every Fano manifold with Condition~\ref{cond:*} would be a rational homogeneous manifold.
However, this speculation is not true in general.
In fact, we clarify that the Fano $7$-fold $X_0$ constructed by Ottaviani \cite{Ott88,Ott90} satisfies the following properties (see Theorem~\ref{thm:notnef} for details):

\begin{enumerate}
\item $X_0$ is a Fano $7$-fold with Picard number two which admits two different smooth $\bP^2$-fibrations $\pi$ and $p$  over the five dimensional quadric:
\begin{align}\label{rem:counterexample}
\xymatrix{
       & X_0 \ar[ld]_{\pi} \ar[rd]^{p} &        \\
 \bQ^5 &                               & \bQ^5. \\
}
\end{align}
In particular, $X_0$ satisfies Condition~\ref{cond:*}.
\item The tangent bundle $T_{X_0}$ is not nef.
In particular, $X_0$ is not homogeneous.
\end{enumerate}
Moreover, we show that a Fano $7$-fold with property (1)  is unique up to isomorphism (for a more precise statement, see Theorem~\ref{thm:char}).
In particular, the existence of $X_0$ shows that smooth Fano $n$-folds with $\rho_X > n-6$ is not necessarily rational homogeneous.

The purpose of this paper is to classify Fano $n$-folds with Condition~\ref{cond:*} and Picard number $\rho _X >n-6$:

\begin{theorem}\label{thm:rhon-6}
Let $X$ be a Fano $n$-fold with Condition~\ref{cond:*} and $\rho_X > n-6$.
Then $X$ is either
\begin{enumerate}
 \item  a rational homogeneous manifold or
 \item  $\pbP{1}{n-7} \times ({X_0} \text{ as in \eqref{rem:counterexample}})$. \label{exc}
\end{enumerate}
\end{theorem}

\medskip
By a theorem of Demailly, Peternell and Schneider \cite[Theorem~5.2]{DPS94} (see also \cite[Theorem~4.4]{SW04} and \cite[Proposition~4]{MOSW15}), a Fano manifold $X$ with nef tangent bundle satisfies Condition~\ref{cond:*} if one assumes Conjecture~\ref{conj:CP} for $k$-folds with Picard number one and $k \leq \dim X - \rho _X +1$.
Note that Conjecture~\ref{conj:CP} is true in dimension at most five.
Hence, as a corollary of Theorem~\ref{thm:rhon-6}, we obtain a result with respect to Conjecture~\ref{conj:CP}:

\begin{corollary}\label{cor:CPn-6}
If Conjecture~\ref{conj:CP} for $6$-folds with Picard number one is true,  then Conjecture~\ref{conj:CP} is true for $n$-folds with $\rho_X > n-6$.
\end{corollary}

\medskip
Note that the above example $X_0$ as in \eqref{rem:counterexample} also gives a negative answer to  the following problem for $q=1$ (cf.\ \cite{Yas12,Yas14}):
\begin{problem}[{\cite[Problem~6.4]{CP92}}]
\label{prob:CP}
Let $X$ be a Fano manifold.
If ${\bigwedge}^{q}T_X$ is nef on every extremal rational curve, then is ${\bigwedge}^{q}T_X$ nef?
\end{problem}

A significant progress concerning Conjecture~\ref{conj:CP} and Problem~\ref{prob:CP} for $q=1$ is obtained by Mu\~noz, Occhetta, Sol\'a-Conde, Watanabe and  Wi\'sniewski \cite{MOSW15,OSWW14};
They show that, if every elementary contraction of a Fano manifold $M$ is a smooth $\bP^1$-fibration, then $M$ is a complete flag manifold.
In particular, Problem~\ref{prob:CP} for $q=1$ and Conjecture~\ref{conj:CP} are affirmative for such Fano manifolds.
For further results or background materials on Conjecture~\ref{conj:CP}, we refer the reader to the article \cite{MOSWW14}.

We explain an outline of this paper:
In Sect.~1, we present some generality on intersection numbers with the relative anticanonical divisor on a projectivised vector bundle.
To study such intersection numbers, we introduce two invariants $d_i(\sE)$ and $\Delta_i (\sE)$ of a vector bundle $\sE$ as a variant of the definition of Segre classes and Chern classes.
The results are used later to study Fano bundles of rank bigger than three.
In Sect.~\ref{sec:char}, we give two descriptions of the manifold ${X_0}$ as in  \eqref{rem:counterexample}
and then a characterization of ${X_0}$  is established.
Here, as in \cite{MOS14,Wat14b}, slopes for Fano bundles (see Definition~\ref{def:slope}) and numerical conditions on slopes play important roles (cf.\ \cite[Section~2]{Kan15b}).
In Sect.~\ref{sec:cont}, we generalize some results in \cite{DPS94,SW04,MOSW15,Kan15b} for Fano manifolds with nef tangent bundle to those for Fano manifolds with Condition~\ref{cond:*}.
Once we obtain these generalization of the results, we may prove Theorem~\ref{thm:rhon-6} for $\rho _X > n-5$ by similar arguments as in \cite{CP93,Wat14a,Wat15,Kan15b}.
For this we only give a sketch of the proof.
In Sect.~\ref{sec:proof}, we complete the proof of Theorem~\ref{thm:rhon-6} for $\rho _X = n-5$ and Corollary~\ref{cor:CPn-6}.

\begin{convention}\label{conv}
We work over the field of complex numbers.
Given a vector bundle $\sE$ on a manifold $Y$, we will denote by $\bP_Y(\sE) =\bP(\sE)$ the Grothendieck projectivization of the vector bundle, and a morphism is called a \emph{$\bP^r$-bundle} if it is isomorphic to the projection of some projectivised vector bundle.
On the other hand, a \emph{smooth $\bP^r$-fibration} is a smooth morphism whose fibers are isomorphic to $\bP^r$.

We will denote by $\sN $ the null-correlation bundle on $\bP ^3$,  by $\sS$ the spinor bundle on $\bQ ^{2n+1}$, by $\sS_i$ $(i=1,2)$ the spinor bundles on $\bQ ^{2n}$  and by $\sC $ the Cayley bundle on $\bQ ^5$.
For the definitions of the null-correlation bundle, the spinor bundles and the Cayley bundle, we refer the reader to \cite{OSS80,Ott88,Ott90}.
\end{convention}

\begin{acknowledgements}
The author wishes to express his gratitude to his supervisor Professor Hiromichi Takagi for his encouragement, comments and suggestions.
The author is also grateful to Professors J\'anos Koll\'ar, Keiji Oguiso and  Kiwamu Watanabe for helpful comments and suggestions.
The author also wishes to thank Doctors Sho Ejiri and Takeru Fukuoka for helpful discussions.
The author is a JSPS Research Fellow and he is supported by the Grant-in-Aid for JSPS fellows (JSPS KAKENHI Grant Number 15J07608).
This work was supported by the Program for Leading Graduate Schools, MEXT, Japan.
\end{acknowledgements}

\section{Preliminaries: classes $d_i(\sE)$ and $\Delta_i (\sE)$ for a vector bundle $\sE$}\label{sec:pre}
In this section we introduce two invariants $d_i(\sE)$ and $\Delta_i (\sE)$ of a vector bundle $\sE$ as a variant of the definition of Segre classes and Chern classes.
Before the definition, we review the definition of Segre classes and Chern classes.
For more details we refer the reader to \cite{Ful98}.
Note that our $\bP(\sE)$ is $P(\sE^*)$ in \cite{Ful98} and our odd Segre classes differ in sign from those in \cite{Ful98}. 
Let $\sE$ be a vector bundle of rank $r$ on a projective manifold $Y$.
Then the $i$-th Segre class $s_i(\sE)$ is defined by the equation
\[
s_i(\sE)=\pi_* \left(\xi_\pi^{r-1+i} \right),
\]
where $\pi \colon \bP(\sE) \to Y$ is the natural projection and $\xi_\pi$ is the tautological divisor.
Then the $i$-th Chern class $c_i(\sE)$ is  defined to be the $i$-th coefficient of $\left(\sum^{\infty}_{i=0} (-1)^i s_i(\sE) t^i\right)^{-1}$.
It is well known that Chern classes vanish for $i>r$.
Hence, by the equations $c_i(\sE) = 0$ for $i>r$, we can describe $s_i(\sE)$ for $i>r$ explicitly by $s_1(\sE), \dots , s_r(\sE)$.
Therefore on the projectivised vector bundle $\bP(\sE)$ of dimension $n$, the intersection number $\xi_\pi^{n-i}\cdot\pi^*D_1 \cdots \pi^*D_i$ is expressed in terms of intersections between $s_1(\sE), \dots , s_r(\sE)$ and $D_1 , \dots ,D_i$ for $0 \leq i \leq n$.

In this manner Segre classes and Chern classes are suitable to describe intersection number with the tautological divisor.
In some cases, however, it is not comfortable to study the geometry of a projectivised vector bundle $\bP(\sE)$ with Segre classes and Chern classes because the classes vary if we twist the bundle with a line bundle.
To avoid this, we use the relative anticanonical divisor $-K_\pi$ instead of the tautological divisor $\xi_\pi$.
The relative anticanonical divisor $-K_\pi$ or the ``\emph{normalized hyperplane class} ${-K_{\pi}}/{r}$'' of a projectivised bundle is used effectively for the first time in Miyaoka's work  
\cite{Miy87} (cf \cite{Yas11}).
Basically the same one of the following definition is also included in \cite[\S6.b]{Nak04}.

\subsection{Definition of $d_i(\sE)$ and $\Delta_i (\sE)$}
Let $Y$ be a smooth projective variety of positive dimension and $\sE$ a vector bundle of rank $r$ on $Y$.
Set $X \coloneqq \bP (\sE)$ and let $\pi \colon X \to Y$ be the natural projection.
We will denote by $\xi_\pi$ the class of the tautological divisor on $\bP(\sE)$.
Then we have $-K_\pi = r\, \xi_\pi - \pi^* c_1 (\sE)$, where $-K_\pi$ is  the relative anticanonical divisor for $\pi$. Let $n$ be the dimension of $X$.

By the definition of Segre classes, we have
\begin{align}\label{eq:1}
  \pi_* \left((-K_\pi)^{r-1+i}\right) = r^{r-1} \sum^{i}_{j=0} \binom{r-1+i}{r-1+j}r^j\,s_j(\sE)s_1(\sE)^{i-j}.
\end{align}
Motivated by this, we define the classes $d_i(\sE)$ and $\Delta_i (\sE)$ as follows:
\begin{definition}
Let the notation be as above.
\begin{enumerate}
 \item
 Set
 \[
 d_i (\sE) \coloneqq \sum^{i}_{j=0} \binom{r-1+i}{r-1+j}r^j\,s_j(\sE)s_1(\sE)^{i-j}.
 \]

 \item
 Set
 \[
 d_t(\sE) \coloneqq \sum^{\infty}_{i=0} d_i(\sE) t^i
 \]
 and
 \[
 \Delta_t(\sE) \coloneqq d_{-t}(\sE)^{-1}.
 \]
 Then $\Delta _i(\sE)$ is defined to be the $i$-th coefficient of $\Delta_t(\sE)$.
\end{enumerate}
\end{definition}

\begin{remark}\label{rem:di}\hfill
\begin{enumerate}
 \item \label{rem:di1}
 We have $\pi_* \left((-K_\pi)^{r-1+i} \right) = r^{r-1} d_i(\sE)$ by \eqref{eq:1}.

 \item \label{rem:di2} By \ref{rem:di1}, we have the following  for $D_1$, $D_2 \in N^1 (Y)$:
 \[
 (-K_\pi + \pi^* D_1)^i\cdot \pi^*D_2^{n-i} = r^{r-1}\sum^i_{k=0} \binom{i}{i-k} d_{k+1-r}(\sE)\cdot D_1^{i-k} D_2^{n-i}.
 \]
 
\item \label{ex:dD}
For later usage, we write down the first few $d_i(\sE)$ explicitly:
\hfill
\begin{enumerate}[label=(\alph*)]
 \item $d_0(\sE) =1$,
 \item $d_1(\sE) = 0$,
 \item $d_2(\sE) = \dfrac{r(r-1)}{2}c_1(\sE)^2-r^2\,c_2(\sE) = r\varDelta$, where $\varDelta$ is the discriminant of the vector bundle $\sE$,
 \item $d_3(\sE) = \dfrac{r(r-1)(r-2)}{3}c_1(\sE)^3-r^2(r-2)\,c_1(\sE)c_2(\sE) + r^3\, c_3(\sE)$.
\end{enumerate}
 
\item \label{rem:di3} By definition, we have
\begin{align}\label{eq:Delta}
 \Delta_i(\sE) = \sum_{\substack{j_1 + \cdots + j_k = i, \\  j_l>0}} (-1)^{i-k} d_{j_1}(\sE)\cdots d_{j_k}(\sE).
\end{align}
Hence the first few $\Delta_{i} (\sE)$ are written down explicitly as follows:
\begin{enumerate}[label=(\alph*)]
 \item $\Delta_0(\sE) = 1$,
 \item $\Delta_1(\sE) = 0$,
 \item $\Delta_2(\sE) = -d_{2}(\sE)$,
 \item $\Delta_{3}(\sE) = d_{3}(\sE)$,
 \item $\Delta_{4}(\sE) = -d_{4}(\sE)+d_{2}(\sE)^{2}$,
 \item $\Delta_{5}(\sE) = d_{5}(\sE)-2\,d_{2}(\sE)d_3(\sE)$.
\end{enumerate}
\end{enumerate}
\end{remark}

We establish a vanishing of $\Delta_i(\sE)$ and ``Grothendieck's relation'' for $-K_\pi$ in the next proposition.
\begin{proposition}\label{prop:vanishing}
$\Delta_i(\sE) = 0 $ for $i>r$ and
\[
\sum_{i=0}^{r} (-1)^i  (-K_\pi)^{r-i} \pi ^* \Delta_i (\sE)=0.
\]
\end{proposition}

\begin{proof}
Set 
\[
\widetilde\Delta_i (\sE) \coloneqq
\begin{dcases}
 \sum^{i}_{k=0} (-1)^{i-k}\binom{r-k}{i-k}r^k c_k(\sE)c_1(\sE)^{i-k} &\text{if $i \leq r$,}\\
 0 & \text{if $i>r$}.
\end{dcases}
\]

Note that
\begin{align*}
a_{k,j} \coloneqq
 \sum_{k\leq i \leq j}  (-1)^{i+j-k} \binom{r-i}{r-j} \binom{r-k}{i-k} =
  \begin{cases}
   (-1)^j & \text{if $k=j$,} \\
        0 & \text{otherwise.} 
  \end{cases}
\end{align*}

By a direct calculation, we have
\begin{align*}
&\sum_{i=0}^{r} (-1)^i  (-K_\pi)^{r-i} \pi ^* \widetilde\Delta_i (\sE) \\
&=\sum_{0 \leq k \leq j \leq r} a_{k,j} r^{r-j+k} \xi_\pi ^{r-j} \pi^* \left(c_1(\sE)^{j-k} c_k(\sE)\right) \\
&=r^r\sum_{i=0}^{r} (-1)^i \xi_\pi^{r-i} \pi ^* c_i (\sE).
\end{align*}
Hence it is zero by the usual Grothendieck relation.

Therefore, for every nonnegative integer $m$, we have
\[
\sum_{i=0}^{r} d_{m+1-i}(\sE) \cdot (-1)^i \widetilde\Delta_i (\sE)=0.
\]
This implies that $\Delta_i(\sE) = \widetilde\Delta_i(\sE)$.
This completes the proof.
\end{proof}

\begin{remark}\label{ex:r=3}
By the above proposition and Remark~\ref{rem:di}~\ref{rem:di3}, $d_i(\sE)$ for $i>r$ is described by $d_2(\sE) \dots, d_r(\sE)$  (note that $d_1(\sE)=0$).
For example, if $r=3$, we have 
\[
d_4(\sE) = d_2(\sE) ^2 \text{ and }d_5(\sE) = 2\,d_2(\sE) \cdot d_3(\sE).
\]
\end{remark}

\subsection{Slopes of Fano bundles}\label{subsec:slope}
In this subsection, we assume that $\rho_Y =1$ and $\sE$ is a Fano bundle, i.e.\ $\bP(\sE)$ is a Fano manifold.
Then $Y$ is also a Fano manifold by \cite[Theorem~1.6]{SW90} or \cite[Corollary~2.9]{KMM92a}.
We write $\Pic (Y) = \bZ \, H_Y$ with the ample generator $H_Y$.

\begin{definition}{\cite[Definition~2.1]{MOS14}}\label{def:slope}
 The \emph{slope} $\tau$ for the pair $(Y,\sE)$ is the real number $\tau$ such that $-K_\pi + \tau\, \pi^*H_Y$ is nef but not ample.
\end{definition}

Then, by \cite[Proposition~2.4]{MOS14}, \cite[Corollary~2.8]{KMM92a}, the Kawamata rationality theorem and the Kawamata-Shokurov base point free theorem \cite{KMM87,KM98}, we have the following:
\begin{proposition}[{\cite[Proposition~2.4 and Remark~2.9]{MOS14}}] \label{prop:tau}\hfill
\begin{enumerate}
 \item $\tau = 0$ if and only if $X \simeq \bP^{r-1}\times Y$.
 \item $0 \leq \tau < r_Y $, where $r_Y$ is the Fano index of $Y$.
 \item $\tau \in \bQ$.
 \item $-K_\pi + \tau \, \pi^*H_Y$ is semiample and defines another contraction $p \colon X \to Z$.
\end{enumerate}
\end{proposition}

Then we have $\kappa (-K_\pi + \tau \, \pi^*H_Y) = \dim Z$, where $\kappa (-K_\pi + \tau \, \pi^*H_Y)$ is the Kodaira dimension of $-K_\pi + \tau \, \pi^*H_Y$.
In particular $(-K_\pi + \tau \, \pi^*H_Y)^i\cdot \pi^*H_Y^{n-i} = 0$ for $i> \kappa (-K_\pi + \tau \, \pi^*H_Y)$.
Hence we have the following  by  Remark~\ref{rem:di}~\ref{rem:di2}:

\begin{proposition}\label{prop:eq}
For $i> \kappa (-K_\pi + \tau \,\pi^*H_Y)$, we have
\[
\sum^i_{k=0} \binom{i}{i-k} d_{k+1-r}(\sE)\cdot H_Y^{n-k}\, \tau^{i-k} =0.
\]
\end{proposition}

\section{A characterization of Ottaviani bundle}\label{sec:char}
\subsection{Ottaviani bundle and the family of special planes on the five dimensional quadric}\label{subsec:example}
\begin{definition}\label{def:Ottaviani}
Let $\sE$ be a stable vector bundle of rank three on $\bQ^5$ with Chern classes $(c_1,c_2,c_3)=(2,2,2)$.
Ottaviani \cite[Section~3]{Ott88} shows that such a vector bundle exists and each of which  arises as a quotient of the dual of the spinor bundle on $\bQ^5$:
\[
0 \to \cO_{\bQ ^5} \to \sS^* \to \sE \to 0.
\]
In this paper we call this bundle \emph{Ottaviani bundle} and denote by ${X_0}$ the projectivised Ottaviani bundle $\bP(\sE)$.
\end{definition}

Set $Y \coloneqq \bQ^5$.
The surjection $\sS^* \to \sE$ gives an immersion of projectivised vector bundles $i \colon {X_0} \coloneqq \bP_Y(\sE) \to \bP_Y (\sS^*)$.
By the definition of the spinor bundle, there exists a smooth $\bP^2$-bundle $p' \colon \bP_Y (\sS^*) \to S_3 \simeq \bQ^6$, where $S_3$ is the spinor variety of type $B_3$.
Let $\bP(\sE) \xrightarrow{p} Z \xrightarrow{h} S_3$ be the Stein factorization of $p' \circ i$:
\[
\xymatrix{
         & {X_0}=\bP_Y(\sE) \ar[ld]_(0.5){\pi} \ar[rd]^(0.6){p} \ar@{^{(}->}[r]^(0.55){i} & \bP_Y (\sS^*) \ar[lld]_{\pi '} \ar[rd]^{p'} &                  \\
Y = \bQ ^5 &                                                            & Z \ar[r]^(0.4){h}                              & S_3 \simeq \bQ^6 \\
}
\]

We use the same notation as in Section~2 (e.g.\ $H_Y$ is the ample generator of $\Pic (Y)$).
Note that the Chern classes $(c_1,c_2,c_3,c_4)$ of $\sS ^*$ are $(2,2,2,0)$ and that $p$ is not an isomorphism  since $\dim {X_0} > \dim Z$.

\begin{theorem}\label{thm:notnef}
The following hold:
\begin{enumerate}
 \item $p$ is a $\bP^2$-bundle over the $5$-dimensional quadric $Z \simeq \bQ^5$.
In particular, ${X_0}$ satisfies Condition~\ref{cond:*}.
 \item The tangent bundle of ${X_0}$ is not nef.
\end{enumerate}
\end{theorem}

\begin{proof}
By the definition of the spinor bundle, we have $p'^* \cO_{\bQ^6}(1) \simeq \cO_{\bP (\sS^*)} (1)$ and hence the vector bundle $\sS^*$ is nef but not ample.
Hence the slope for the pair $(Y,\sS^*)$ is two and $p'$ is defined by the semiample divisor $-K_{\pi '}+ 2 \pi'^*H_Y$.
Therefore the morphism $p$ is defined by the semiample divisor $(-K_{\pi'}+ 2 \pi'^*H_Y)|_{{X_0}}$, which is equivalent to $4\, \xi_\pi$.

Because $\dim Z < \dim {X_0}$, the divisor $-K_{\pi}+ 2 \pi^*H_Y = 3 \xi_\pi$ is nef but not ample.
This implies that $\sE$ is a Fano bundle whose slope $\tau$ is two.

By a direct calculation using Remark~\ref{rem:di}~\ref{rem:di2}, Remark~\ref{rem:di}~\ref{ex:dD} and Remark~\ref{ex:r=3}, we have $(-K_\pi + 2 \pi^*H_Y)^6\cdot \pi^*H_Y = 0$
and $(-K_\pi + 2 \pi^*H_Y)^5\cdot \pi^*H_Y^2 \neq 0$.
Hence we have $\dim Z =5$, $h$ is an immersion and $p$ is the base change of $p'$ over $Z$.
Since $p'$ is a $\bP^2$-bundle, $p$ is also a $\bP^2$-bundle.
Furthermore $Z$ is a linear section $\bQ^5$ of $\bQ^6$ since the normal bundle of ${X_0}$ in $\bP(\sS^*)$ is $\cO_{\bP (\sS^*)} (1) |_{{X_0}} \simeq p'^* \cO_{\bQ^6}(1)|_{X_0}$.
Hence ${X_0}$ is a Fano $7$-fold with Picard number two which satisfies Condition~\ref{cond:*}:
\[\xymatrix{
               & {X_0} \ar[ld]_{\pi} \ar[rd]^{p} &                 \\
Y \simeq \bQ^5 &                               & Z \simeq \bQ^5. \\
}\]

By \cite[Example~3.3]{Ott90}, there exists the following exact sequence on $Y = \bQ^5$:
\[
0 \to \sC (1) \to \sE \to \cO_{\bQ^5}(1) \to 0,
\]
where $\sC$ is the Cayley bundle on $\bQ^5$.
Then the surjection $\sE \to \cO_{\bQ^5}(1)$ gives a section $S \simeq \bQ^5 \subset {X_0}$ of $\pi$ with normal bundle $N_{S/{X_0}}\simeq \sC^* \simeq \sC (1)$, which is \emph{not} nef by \cite[Theorem~3.5]{Ott90}.
Hence the tangent bundle of ${X_0}$ is \emph{not} nef since the normal bundle $N_{S/{X_0}}$ is a quotient of the tangent bundle.
\end{proof}

\begin{remark}
In \cite{Pan13}, a smooth projective variety is called convex if  
\[
H^1(\mu ^* T_{X_0}) = 0
\]
for every morphism $\mu \colon \bP^1 \to X_0$, and Pandharipande proved that a convex, rationally connected smooth complete intersection is a homogeneous manifold.  
Note that ${X_0}$ is not convex in the sense of \cite{Pan13}.
Indeed the restriction of $\sC(1)$ on a special line in $\bQ^5 \simeq S$ is $\cO_{\bP^1}(-1)\oplus \cO_{\bP^1}(2)$ \cite[Theorem~3.5]{Ott90}.
Hence if we take a double cover of the special line, we have a morphism $\mu \colon \bP^1 \to X_0$ with $H^1(\mu ^* T_{X_0}) \neq 0$
\end{remark}

\begin{definition}[{\cite[Section~1]{Ott90}}]
Let $\bO$ be the complexified Cayley octonions, which is an algebra generated by $1$, $e_1$, $e_2$, \dots $e_7$ with the following relations:
(1) $e_i^2 =-1$,
(2) $e_i \cdot e_j= - e_j \cdot e_i$ for $i \neq j$,
(3) $e_1 \cdot e_2 = e_3$,
(4) $e_1 \cdot e_4 = e_5$,
(5) $e_1 \cdot e_7 = e_6$,
(6) $e_2 \cdot e_5 = e_7$,
(7) $e_2 \cdot e_4 = e_6$,
(8) $e_3 \cdot e_4 = e_7$,
(9) $e_3 \cdot e_6 = e_5$.

It is known that the automorphism group of $\bO$ is a semisimple group of type $G_2$, and the group acts on the variety of projectivised elements with null-square.
The equations for the variety of projectivised elements with null-square is
\[
x_0=\sum_{i=1}^{7}x_i^2 =0.
\]
Hence it is naturally isomorphic to the five dimensional quadric $\bQ^5$.

The \emph{special plane on $\bQ^5$} through a point $y \in \bQ^5$ is defined to be $\Pi_{y} \coloneqq \{x \in \bQ ^5 \mid\, x\cdot y =0  \}$.

Set $X\coloneqq \{(x,y) \in \bQ^5 \times \bQ ^5 \mid\, x\cdot y =0  \}$ and let $p_1 \colon X \to Y \coloneqq \bQ^5$ be the first projection and $p_2 \colon X \to Z \coloneqq \bQ^5$ the second projection.
We call this $X$ the \emph{family of special planes on $\bQ ^5$}.
\end{definition}

Then every $p_2$-fiber over $y \in Z \simeq \bQ^5$ defines the special plane $\Pi_y \subset Y \simeq \bQ^5$.
Hence $p_2 \colon X \to Z$ is a $\bP^2$-bundle.
By the symmetry, $p_1 \colon X \to Y$ is also a $\bP^2$-bundle.
Hence we have the following:
\begin{proposition}\label{prop:fam}
The family of special planes $X$ satisfies Condition~\ref{cond:*} and admits two different $\bP^2$-bundle over $5$-dimensional quadric $\bQ^5$.
\end{proposition}

In the next subsection, we shall  show that the above two manifolds in Theorem~\ref{thm:notnef} and Proposition~\ref{prop:fam} are isomorphic to each other.

\subsection{A characterization of projectivised Ottaviani bundle}
\begin{theorem}\label{thm:char}
Let $X$ be a Fano $7$-fold with Picard number two which satisfies Condition~\ref{cond:*}.
Assume that $X$ has a smooth $\bP^2$-fibration $\pi \colon X \to Y$.
Then $X$ is isomorphic to $\bP^2 \times Y$ or ${X_0}$ as in Definition~\ref{def:Ottaviani}.

In particular two manifolds in Theorem~\ref{thm:notnef} and Proposition~\ref{prop:fam} are isomorphic to each other.
\end{theorem}

The rest of this section is occupied with our proof of Theorem~\ref{thm:char}.
Let $X$ be a manifold as in Theorem~\ref{thm:char}.
Then $Y$ is a $5$-dimensional rational homogeneous manifold with Picard number one by Condition~\ref{cond:*}.
By the classification of rational homogeneous 5-folds, we have $Y \simeq \bP ^5$, $\bQ ^5$ or $K(G_2)$, where $K(G_2)$ is the $5$-dimensional Fano contact homogeneous manifold of type $G_2$.
Since the Brauer group of $Y$ is trivial there exists a vector bundle $\sE$ over $Y$ such that $X \simeq \bP(\sE)$.
By Condition \ref{cond:*}, we have the other smooth elementary rational homogeneous fibration $p \colon X \to Z$ over some rational homogenous manifold $Z$:
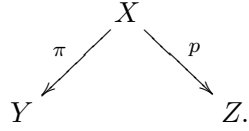
\[
\xymatrix{
  & X \ar[ld]_{\pi} \ar[rd]^{p} &    \\
Y &                             & Z. \\
}\]
If the slope $\tau$ for the pair $(Y,\sE)$ is zero, then $X \simeq \bP^2 \times Y$ by Proposition~\ref{prop:tau}.

Hence, in the rest of this section we assume that the slope $\tau$ for the pair $(Y,\sE)$ is nonzero and we shall show that $X $ is isomorphic to ${X_0}$ as in Definition~\ref{def:Ottaviani}.

\begin{notation}\label{not:chow}
In any case, we have $A^i(Y)_\bZ \simeq \bZ$ for each $i$.
We fix an effective generator of $A^i(Y)_{\bZ}$ as follows:
$A^0(Y)_\bZ = \bZ \,[Y]$,
$A^1(Y)_\bZ = \bZ \,H_Y$,
$A^2(Y)_\bZ = \bZ \,\Sigma _Y$,
$A^3(Y)_\bZ = \bZ \,P_Y$,
$A^4(Y)_\bZ = \bZ \,\ell _Y$,
$A^5(Y)_\bZ = \bZ \,\text{\{point\}}$.
We sometimes identify each class in $A^i(Y)_\bZ$ with some integer.
Therefore there exist a triple of integers $(n_Y,m_Y,d_Y)$ which satisfies $H_Y^2 = n_Y\, \Sigma_Y$, $H_Y \cdot \Sigma _Y = m_Y\, P_Y$ and $H_Y^5 =d_Y$.

We will write $d_2(\sE) = a \cdot H_Y^2$ and $d_3(\sE) = b \cdot H_Y^3$ with rational numbers $a$ and $b$.
Note that $d_2(\sE) = n_Y\,a \in \bZ$ and $d_3(\sE) = n_Ym_Y\,b \in \bZ$.
\end{notation}

\begin{remark}\label{rem:nmd}
 We have the following:
\begin{enumerate}
 \item $(n_Y,m_Y,d_Y)=(1,1,1)$ if $Y \simeq \bP^5$,
 \item $(n_Y,m_Y,d_Y)=(1,2,2)$ if $Y \simeq \bQ^5$,
 \item $(n_Y,m_Y,d_Y)=(3,2,18)$ if $Y \simeq K(G_2)$.
\end{enumerate}
\end{remark}

\begin{lemma}
$\dim Z \leq 5$. 
\end{lemma}

\begin{proof}
Assume to the contrary $\dim Z =6$.
Then  the other contraction $p$ is a $\bP^1$-bundle over $Z$.
Then, since $b_i(X)= b_{i-4}(Y) + b_{i-2}(Y) + b_{i}(Y)  $ and $b_i(X)=  b_{i-2}(Z) + b_{i}(Z)  $ , we have $b_4(Z)=2$ and $b_6(Z)=1$.
 This contradicts the hard Lefschetz theorem.
\end{proof}

Hence, by  Remark~\ref{ex:r=3} and Proposition~\ref{prop:eq}, we have the following:
\begin{align}
f(\tau) = 21 \, \tau^5 + 35 a \, \tau^3 + 21 b \, \tau^2 + 7 a^2 \, \tau + 2 a b =0, \label{eq:num1}\\
g(\tau) =  15 \, \tau^4 + 15 a \, \tau^2 + 6 b \, \tau + a^2  =0. \label{eq:num2}
\end{align}

Then we have $R(f,g)=0$, where $R(f,g)$ is the resultant.
This is equivalent to
\begin{align}\label{eq:syl}
 0= 9 a (216 \, {b} ^2 + 49 \, {a}^3)(250047\, {b}^4 -222804 \,{a} ^3 {b}^2 +132496\, {a}^6).
\end{align}

\begin{lemma}\label{lem:num}
The following hold:
\begin{enumerate}
 \item \label{lem:num1} $a = -6\, k^2$  and $b = 7\, k^3$ for some $0\neq k \in \bZ$.  
 \item \label{lem:num2} $\tau =2k$.
 \item \label{lem:num3} Up to twisting $\sE$ with a line bundle, we have $\tau=c_1(\sE)$ and the following possibilities for $(Y; c_1(\sE), c_2(\sE), c_3(\sE))$:
  \begin{enumerate}[label=(\alph*),font=\normalfont]
   \item \label{lem:num3a} $(\bP ^5;2,2,1)$,
   \item $(\bP ^5;4,8,8)$,
   \item \label{lem:num3c} $(\bQ ^5;2,2,2)$,
   \item  $(\bQ ^5;4,8,16)$,
   \item  $(K(G^2);2,6,6)$.
  \end{enumerate}
 \item \label{lem:num4} $\dim Z =5$.
\end{enumerate}
\end{lemma}

\begin{proof}
\ref{lem:num1} First observe that $a \neq 0$.
Indeed if $a=0$ then the equation $f(\tau)= g(\tau)=0$ gives $\tau = 0$, which contradicts our assumption $\tau \neq 0$.
Hence the equation~\eqref{eq:syl} gives 
\[
216 \, {b} ^2 + 49 \, {a}^3=0\] or \[250047\, {b}^4 -222804 \,{a} ^3 {b}^2 +132496\, {a}^6=0.
\]
For the latter equation we have $b \not \in \bQ$, which is a contradiction.
Hence the former case occurs, equivalently we have:
\[
216\, n_Y{d_3}^2 + 49 \, m_Y^2{d_2}^3 =0.
\]
Note that $d_2$, $d_3\in \bZ$ and $(n_Y,m_Y)$ is described as in Remark~\ref{rem:nmd}.
For each case we can solve the equation and the first assertion follows.

\ref{lem:num2} By \ref{lem:num1} and the equations \eqref{eq:num1} and \eqref{eq:num2}, we have
\begin{align*}
 (\tau-2k)^2(\tau+k)(\tau^2+3\, k\tau -k^2)=0,\\
 (\tau-2k)(5\, \tau^3 + 10 \, k\tau^2 - 10\, k^2\tau - 6\, k^3)=0.
\end{align*}
This gives the second assertion.

\ref{lem:num3}
By Proposition~\ref{prop:tau} and our assumption $\tau \neq 0$, we have $0< \tau =2k < r_Y$.
Hence we have $k=1$ if $Y \simeq K(G_2)$, or $k=1$, $2$ otherwise.
Since the rank of $\sE$ is three, we may assume that $1 \leq c_1(\sE) \leq 3$ if $k=1$, and that $4 \leq  c_1(\sE) \leq 6$ if $k=2$.

By \ref{lem:num1} and Remark~\ref{rem:di}~\ref{ex:dD},
the following hold:
\begin{align*}
-6\, k^2 &=  3\, c_1(\sE)^2 - \frac{9}{n_Y}\, c_2(\sE), \\
7\, k^2 &=  2\,c_1(\sE)^3 - \frac{9}{n_Y}\,c_1(\sE)c_2(\sE) + \frac{27}{n_Ym_Y}\,c_3(\sE).
\end{align*}
We have the assertion by solving these for each case.

\ref{lem:num4} Since $\tau=2\,k$, we have $k \neq 0$.
By a direct calculation with Remark~\ref{rem:di} and  Remark~\ref{ex:r=3}, we obtain 
\[
(-K_\pi + \tau \, \pi^*H_Y)^5 \pi^*H_Y^2= 54 k^2 \neq 0.
\]
Hence $\dim Z = 5$.
\end{proof}

By Lemma~\ref{lem:num}~\ref{lem:num4} and Condition~\ref{cond:*}, $Z$ is also a rational homogeneous $5$-fold with Picard number one and $p$ is an elementary rational homogeneous fibration of relative dimension two.
Hence  $Z$ is isomorphic to $\bP^5$, $\bQ^5$ or $K(G_2)$, and $p$ is a smooth $\bP^2$-fibration by the classification of rational homogeneous manifolds.
Since the Brauer group of $Z$ is trivial, $p$ is a $\bP ^2$-bundle.
Hence there exists a rank three vector bundle $\sF$ on $Z$ such that $X \simeq \bP_Z(\sF)$.
Therefore we have the following diagram:
\[\xymatrix{
        & \bP_Y(\sE) = X = \bP_Z(\sF) \ar[ld]_{\pi} \ar[rd]^{p} &          \\
(Y,\sE) &                                                   & (Z,\sF), \\
}\]
where $p$ is the natural projection.
Twisting with a line bundle, we may assume that $(Z,\sF)$ also satisfies the conditions in Lemma~\ref{lem:num}. 

We use a similar notation for $A^i(Z)$ as in Notation~\ref{not:chow}.

\begin{lemma}
Only Lemma~\ref{lem:num}~\ref{lem:num3}~\ref{lem:num3c} occurs.
\end{lemma}

\begin{proof}
In any case of Lemma~\ref{lem:num}~\ref{lem:num3}, $\sE$ and also $\sF$ are nef but not ample since the slope $\tau$ for each bundle is equals to its first Chern class.
Hence $\xi_\pi= p^*H_Z$ and $\xi_p = \pi^* H_Y$.
 
By the Grothendieck relation,
\[
n_Z m_Z \, p^* P_Z= p^*H_Z^3
= \xi_\pi^3
= c_1(\sE)\, \xi_\pi^2 \cdot \pi^*H_Y - c_2(\sE)\, {\xi_\pi}\cdot \pi^*\Sigma_Y +c_3(\sE)\pi^* P_Y.
\]
We also have the following:
\begin{align*}
\xi_p \cdot p^* \Sigma_Z = n_Z \,\xi_\pi^2 \, \pi^* H_Y,\\
\xi_p^2 \cdot p^* H_Z = n_Y \,{\xi_\pi} \,\pi^* \Sigma_Y.
\end{align*}
Because the triples $(p^*P_Z ,\, \xi_p \cdot p^* \Sigma_Z ,\, \xi_p^2 \cdot p^* H_Z)$ and $(\pi^*P_Y ,\, \xi_\pi \cdot \pi^* \Sigma_Y ,\, \xi_\pi^2 \cdot \pi^* H_Y)$ are $\bZ$-bases of $A^3(X)$, the following matrix is unimodular:
\[
\begin{pmatrix}
 0&n_Y&0\\
 n_Z&0&0\\
 \dfrac{c_1(\sE)}{n_Zm_Z} & -\dfrac{c_2(\sE)}{n_Zm_Z} & \dfrac{c_3(\sE)}{n_Zm_Z}
\end{pmatrix}.
\]

From this it follows that 
\begin{enumerate}
 \item Lemma~\ref{lem:num}~\ref{lem:num3}~\ref{lem:num3a} occurs and $Z \simeq \bP^5$ or
 \item Lemma~\ref{lem:num}~\ref{lem:num3}~\ref{lem:num3c} occurs and $Z \simeq \bQ^5$.
\end{enumerate}
However the first one cannot happen by \cite{Sat85}.
\end{proof}

The following completes our proof of Theorem~\ref{thm:char}.
\begin{lemma}
Let $(Y,\sE)$ be as in Lemma~\ref{lem:num}~\ref{lem:num3}~\ref{lem:num3c}, then $\sE$ is stable.
\end{lemma}
\begin{proof}
It is enough to show that $H^0(\sE (-1))= H^0(\sE ^*)=0$.
 
Because $\xi_\pi$ defines the another contraction of fiber type $p \colon X \to Z$, we have $\xi_\pi \in \Psef(X) \setminus \Bigcone (X)$, where $\Psef(X)$ is the pseudoeffective cone of $X$ and $\Bigcone (X)$ the big cone of $X$.
Hence we have $0 = H^0\left(\cO (\xi_\pi - \pi^*H_Y)\right) = H^0 (\sE(-1))$.

Note that $\sE$ is a 2-ample vector bundle because $p$ is a $\bP^2$-bundle.
Hence $H^0(\sE ^*) = H^5(\omega_Y \otimes \sE) = 0$ by the Sommese vanishing theorem \cite[Proposition~1.14]{Som78}.
\end{proof}

Hence $\sE$ is Ottaviani bundle in Definition~\ref{def:Ottaviani}.
This completes the proof of Theorem~\ref{thm:char}.

\section{Contractions of Fano manifolds with Condition~\ref{cond:*} and Proof of Theorem~\ref{thm:rhon-6} in case $\rho_X > n-5$}\label{sec:cont}

\subsection{Contractions of Fano manifolds with Condition~\ref{cond:*}}
In this subsection, we generalize some results known for Fano manifolds with nef tangent bundle to those for Fano manifolds with Condition~\ref{cond:*}.
We call a Fano manifold with nef tangent bundle a \emph{CP manifold}.

\begin{proposition}[{cf.\ \cite[Theorem~5.2]{DPS94}, \cite[Theorem~4.4]{SW04}, \cite[Proposition~4]{MOSW15} for CP manifolds}]\label{prop:cont}
Let $X$ be a Fano manifold with Condition~\ref{cond:*} and $\pi  \colon X \to Y$ a contraction.
Then the following hold:

\begin{enumerate}
 \item \label{prop:cont1} The morphism $\pi$ is smooth and $Y$ is a Fano manifold with Condition~\ref{cond:*}.
 \item \label{prop:cont3} $\rho _F =\rho _X - \rho _Y$  and $j_* \big(\NE (F)\big) = \NE (X) \cap j_* \big( N _1 (F) \big)$ for a $\pi$-fiber $F$, where $j\colon F \to X$ is the inclusion.
 \item \label{prop:cont4} $\NE (X)$ is simplicial.
 \item \label{prop:cont5} The fibers of $\pi$ are Fano manifolds with Condition~\ref{cond:*}.
\end{enumerate}
\end{proposition}

\begin{proof}
\ref{prop:cont1} By induction on $\rho(X/Y)$, we may reduce to the case $\rho(X/Y)=1$.
The first assertion follows from the definition of \ref{cond:*}.
Hence $Y$ is a Fano manifold by \cite[Corollary~2.9]{KMM92a}.
Then, since $X$ satisfies Condition~\ref{cond:*}, $Y$ also satisfies Condition~\ref{cond:*}.

\ref{prop:cont3}, \ref{prop:cont4} 
Note that $T_X$ is $g$-nef for every elementary contraction $g$.
These follow from the same argument as in \cite[Proposition~4]{MOSW15}.

\ref{prop:cont5} By adjunction, $F$ is a Fano manifold.
By \ref{prop:cont3}, every elementary contraction of $F$ is induced by the elementary contraction of $X$.
Hence the assertion follows by induction on $\rho(X/Y)$.
\end{proof}

\begin{proposition}[cf.\ {\cite[Proposition~5]{MOSW15} for CP manifolds}] \label{prop:iniFT}
Let $X$ be a Fano manifold with Condition~\ref{cond:*}.
Assume that there exists a contraction $\pi\colon  X \to M$ onto a Fano manifold $M$ whose elementary contractions are smooth $\bP^1$-fibrations.
Then $X \simeq F \times M$ and $\pi$ is the second projection, where $F$ is a fiber of $\pi$.
\end{proposition}
Note that in the above proposition $M$ is a complete flag manifold by \cite{OSWW14}.

\begin{proof}[Proof of Proposition~\ref{prop:iniFT}]
The same argument in the proof of \cite[Proposition~5]{MOSW15} does work in this case.
\end{proof}

\begin{proposition}[cf.\ {\cite[Theorem~4.1]{Kan15b} for CP manifolds}]\label{prop:2rho}
Let $X$ be a Fano $n$-fold with Condition~\ref{cond:*}.
If $n \leq 2 \rho _X + 1$, then one of the following holds:
\begin{enumerate}
\item $X \simeq Y \times M$, where $Y$ is a Fano manifold with Condition~\ref{cond:*} and $M$ is a complete flag manifold.
\item $X \simeq \pbP{2}{\rho_X}$, $\pbP{2}{\rho_X -1} \times \bP ^3$, $\pbP{2}{\rho_X -1} \times \bQ ^3$, $\pbP{2}{\rho_X - 2} \times \bP (\sS _i)$ or  $\pbP{2}{\rho_X - 2} \times \bP (T_{\bP^3})$.
In particular $X$ is homogeneous in this case.
\end{enumerate}
\end{proposition}

\begin{proof}
The proof of \cite[Theorem~4.1]{Kan15b} is done by induction on $n$ and proceeded as follows:
First we show that every CP manifold $X$ with $n< 2 \rho_X$ admits a contraction onto a Fano manifold $M$ whose elementary contractions are smooth $\bP^1$-fibrations.
Then by \cite[Proposition~5]{MOSW15} and \cite{OSWW14} we have $X \simeq Y \times M$ for a CP manifold $Y$ and  the first case as in this proposition occurs.
Hence we may assume that $n = 2 \rho_X $ or $2 \rho _X +1$ and $X$ does not admit a contraction onto a complete flag manifold.
Then, by induction on $n$, we can reduce to the case of $n\leq 5$, and the assertion follows from \cite{CP91,CP93,Wat14a}.

The proof of this proposition is also proceeded by induction on $n$ and,  once the assertion in the case of $n=4$ or $5$ with Picard number two is proved, then the same argument in the proof of \cite[Theorem~4.1]{Kan15b} works (note that the assertion in the case of $\rho_X=1$ is trivial by our definition of Condition~\ref{cond:*}).
On the other hand, the same argument to classify CP $n$-folds with $n=4$ or $5$ and Picard number two in \cite{CP93,Wat14a} does work to deduce the assertion in the case of $n=4$ or $5$ with Picard number two.
Here we sketch the argument shortened by using the result of \cite{OSWW14}.

Let $X$ be a Fano $n$-fold with Picard number two which satisfies Condition~\ref{cond:*}.
Then there exist two elementary contractions $\pi \colon X \to Y$ and $p \colon X \to Z$:
\[
\xymatrix{
  & X \ar[ld]_{\pi} \ar[rd]^{p} &    \\
Y &                             & Z. \\
}\]
By Condition~\ref{cond:*} and Proposition~\ref{prop:cont}, all fibers of $\pi$ and $p$, $Y$ and $Z$ are rational homogeneous manifolds with Picard number one.
We have $\dim Y + \dim Z \geq \dim X$, and we may assume that $\dim Y \geq \dim Z$.

\subsection*{The case $n=4$}
Then we have (a) $\dim Y = 3$ or (b) $\dim Y = \dim Z = 2$.

Assume $\dim Y =3$.
If $\dim Z=3$, then by \cite{OSWW14} $X$ is a complete flag manifold, hence it is homogeneous.
If $\dim Z \leq 2$, then $X \simeq \bP^1 \times Y$ by the classification of Fano bundles of rank two on $\bP^3$ and $\bQ^3$ \cite{SW90,SSW91}.

Assume $\dim Y =\dim Z =2$.
Then by \cite[Lemma~4.1]{NO07} $X \simeq \bP^2 \times \bP^2$.

\subsection*{The case $n=5$}
Then we have (a) $\dim Y = 4$ or (b) $ \dim Y =3 \geq \dim Z \geq 2$.

Assume that $\dim Y = 4$.
Note that the projectivization of the stable vector bundle of rank two on $\bQ^4$ with Chern classes $c_1= -1$ and $c_2 = (1,1)$ does not satisfy Condition~\ref{cond:*} (see for instance the proof of \cite[Lemma~3.8]{Wat14a}).
Hence by \cite{APW94} we have $X \simeq \bP^1 \times Y$ or $\bP(\sS_i)$, where $\sS_i$ is one of the spinor bundles on $\bQ^4$.

Assume that $\dim Y =3$.
Then, by \cite[Theorem~2]{OW02} and \cite[Lemma~4.1]{NO07},  $X \simeq \bP^2 \times Y$ or $\bP (T_{\bP^3})$.
\end{proof}

\subsection{The case $\rho_X > n-5$}
We show Theorem~\ref{thm:rhon-6} in the case $\rho_X > n-5$.
\begin{theorem}\label{thm:rhon-5}
Let $X$ be a Fano $n$-fold with Condition~\ref{cond:*} and 
$\rho_X > n-5$. Then $X$ is a rational homogeneous manifold.
\end{theorem}

\begin{proof}
The assertion follows from a similar argument for the classification of CP manifold with $\rho_X >1$ and $\rho _X > n-5$ \cite{CP91,CP93,Wat14a,Wat15,Kan15b}.

Here we include only a sketch of the proof.
For details we refer the reader to \cite{CP91,CP93,Wat14a,Wat15,Kan15b}.

By Proposition~\ref{prop:2rho} and induction on $n$, we may assume that $n>2 \rho_X +1$.
Since $\rho_X > n-5$, we have the case $n=6$ and $\rho _X =2$ or the case $\rho _X =1$.
The assertion in the case of $\rho_X =1$ is trivial by the definition of Condition~\ref{cond:*}.
Therefore it is enough to show the assertion in the case of $6$-folds with Picard number two.
In this case the argument is the same as the classification of CP $6$-folds with Picard number two in \cite[Proposition~2.8, Theorem~2.9 and Theorem~4.3]{Kan15b} except for the proof of \cite[Proposition~2.3~(6)$\implies$(1)]{Kan15b}.
Here we give an alternative argument. We need to show the following:

Let $\sE$ be a vector bundle of rank $r$ over a manifold $Y$ and $\pi \colon \bP(\sE ) \to Y$ the natural projection.
Assume that $\bP(\sE )$ is a Fano manifold with Condition~\ref{cond:*}.
If $\sE$ splits into a direct sum of line bundles, then $\bP (\sE )$ is trivial.
 
As in the proof of \cite[Proposition~2.3~(6)$\implies$(1)]{Kan15b}, we may assume that $\Pic Y =\bZ$ and $\sE = \oplus \cO (a_i)$ for some integers $a_1 \leq \cdots \leq a_r$.
By twisting with a line bundle, we may assume that $a_1= \cdots = a_s=0$  and $a_{s+1} \neq 0$ for some integer $s \geq 1$.
If $s <r$, then the relative tautological divisor $\xi$ is nef and big but not ample, which contradicts the fact that $\bP(\sE)$ satisfies Condition~\ref{cond:*}.
Hence we have $r=s$, which completes the proof.
\end{proof}

\section{The case $\rho_X = n-5$}\label{sec:proof}

In this section, we shall complete our proof of Theorem~\ref{thm:rhon-6} in the case $\rho_X = n-5$.
Before proving it, we include here the classification of Fano manifold  of dimension at most six with Condition~\ref{cond:*} for convenience of the readers.
This is equivalent to the classification of rational homogeneous manifolds of dimension at most six by Theorem~\ref{thm:rhon-5}.

\begin{proposition}\label{prop:rhm}
Fano manifold  of dimension at most 6 with Condition~\ref{cond:*} is one of the following:
\renewcommand{\arraystretch}{1.2}

\medskip
\noindent
\begin{tabularx}{\textwidth}{llX}
$\dim X$ &$\rho_{X}$ & $X$\\
 \hline
$6$
&
$1$
&
$\bP^6$,
$\bQ^6$,
$G(2,5)$
or $LG(3,6)$,\\

&
$2$
&
$\bP ^1  \times \bP ^5$,
$\bP ^1  \times \bQ ^5$,
$\bP ^1  \times K(G_2)$,
$ \bP (\sC )$,
$\bP ^2 \times \bP ^4$,
$\bP ^2 \times \bQ ^4$,
$\pbP{3}{2}$,
$\bP ^3 \times \bQ ^3$
or $\pbQ{3}{2} $,\\
 
&
$3$
&
$\pbP{1}{2}  \times \bP ^4$,
$\pbP{1}{2}  \times \bQ ^4 $,
$\bP ^1  \times \bP (\sS _i)$,
$\bP ^1  \times \bP (T_{\bP ^3} )$,
$\bP ^1  \times \bP ^2 \times \bP ^3$,
$\bP ^1  \times \bP ^2 \times \bQ ^3$,
$\pbP{2}{3}$,
$\bP ^2 \times \bP (\sN )$,
$ F(1,2,3;4)$,
$ \bP (T_{\bP ^2}) \times \bP ^3$
or $\bP (T_{\bP ^2}) \times \bQ ^3$,\\

&
$4$
&
$\pbP{1}{3}  \times \bP ^3$,
$\pbP{1}{3}  \times \bQ ^3$,
$\pbP{1}{2}  \times \bP (\sN )$,
$\pbP{1}{2}  \times \pbP{2}{2}$,
$\bP ^1  \times \bP (T_{\bP ^2}) \times\bP ^2 $
or $ \pbPB{T_{\bP ^2}}{2}$,\\

&
$5$
&
$\pbP{1}{4}  \times \bP ^2$
or $\pbP{1}{3}  \times \bP (T_{\bP ^2})$,\\

&
$6$
&
$\pbP{1}{6}$,\\
\hline

$5$
&
$1$
&
$\bP ^5$,
$\bQ ^5$
or  $K(G_2) $,\\
 
&
$2$
&
$\bP ^{1}  \times \bP ^4$,
$\bP ^{1}  \times \bQ ^4$,
$\bP (\sS _i)$,
$ \bP (T_{\bP ^3} )$,
$ \bP ^2 \times \bP ^3$
or $ \bP ^2 \times \bQ ^3$,\\
 
&
$3$
&
$\pbP{1}{2}  \times \bP ^3$,
$\pbP{1}{2}  \times \bQ ^3$,
$\bP^{1}  \times \bP (\sN )$
$\bP^{1} \times \pbP{2}{2}$
or $\bP (T_{\bP ^2}) \times\bP ^2 $,\\

&
$4$
&
$\pbP{1}{3}  \times \bP ^2$
or $\pbP{1}{2}  \times \bP (T_{\bP ^2})$,\\

&
$5$
&
$\pbP{1}5$,\\
\hline
 
$4$
&
$1$
&
$\bP ^4$
or $ \bQ ^4$,\\
 
&
$2$
&
$\bP^{1}  \times \bP ^3$,
$\bP^{1}  \times \bQ ^3$,
$\bP (\sN )$
or $ \pbP{2}{2}$,\\

&
$3$
&
$\pbP{1}{2}  \times \bP ^2$
or $\bP^{1}  \times \bP (T_{\bP ^2})$,\\

&
$4$
&
$\pbP{1}4$,\\
\hline
 
$3$
&
$1$
&
$\bP ^3$
or $\bQ ^3 $,\\

&
$2$
&
$\bP^{1} \times \bP ^2$,\\

&
$3$
&
$\pbP{1}3$,\\
\hline
 
$2$
&
$1$
&
$\bP ^2$,\\

&
$2$
&
$\pbP{1}2$,\\
\hline

$1$
&
$1$
&
$\bP^{1}$,\\
\end{tabularx}

\renewcommand{\arraystretch}{1}

\medskip
\noindent
where $G(2,5)$ is the Grassmannian of planes in $\bC ^5$, $LG(3,6)$ is the Lagrangian Grassmannian of three dimensional subspaces in $\bC^6$, $F(1,2,3;4)$ is the variety of complete flags in $\bC^4$ (see also Convention~\ref{conv}).
\end{proposition}

First we show the assertion in the case of $7$-folds with Picard number two (Proposition~\ref{prop:7-folds}).

\begin{notation}
Let $X$ be a Fano $7$-fold with Picard number two which satisfies Condition~\ref{cond:*}, and  $\pi \colon X \to Y$ and $p \colon X \to Z$ the elementary contractions:
\[
\xymatrix{
  & X \ar[ld]_{\pi} \ar[rd]^{p} &   \\
Y &                             & Z.\\
}
\]
We have $\dim Y + \dim Z \geq \dim X$ and may assume that $\dim Y \geq \dim Z$.
\end{notation}
We denote by  $F(i,j;k)$  the variety of flags $(V_{i} \subset V_{j} \subset \bC ^{k})$ with $\dim V_{i} =i$ and  $\dim V_{j}=j$.

\begin{proposition}\label{prop:7-folds}
In the above notation, $X$ is one of the following:
\begin{enumerate}
\item $\bP^r \times Y$ or  $\bQ^3 \times Y$, where $Y$ is a rational homogeneous manifold.
\item $F(1,2;5)$ or $F(1,4;5)$.
\item ${X_0}$ as in Definition~\ref{def:Ottaviani}.
\end{enumerate}
In particular, Theorem~\ref{thm:rhon-6} holds in this case.
\end{proposition}

\begin{proof}
By the equations $\dim Y + \dim Z \geq \dim X$ and $\dim Y \geq \dim Z$,
we have $\dim Y \geq 4$.

First assume $\dim Y = 6$.
Then $Y$ is isomorphic to $\bP^6$, $\bQ^6$, Grassmannian $G(2,5)$ or Lagrangian Grassmannian $LG(3,6)$; $X$ is isomorphic to some projectivised vector bundle $\bP (\sE)$ with a Fano bundle of rank two on $Y$ and $\pi$ is the natural projection since the Brauer group of $Y$ is trivial.
Note that the fourth Betti number $b_4(LG(3,6))=1$.
Then, by the classification of Fano bundles of rank two \cite{APW94}, \cite{MOS12} and \cite[Lemma~6.1 and Theorem~6.5]{MOS14}, $\sE$ is a split vector bundle $\cO \oplus \cO(a)$ or the universal subbundle on $G(2,5)$ up to twist with some line bundle.

If $\sE$ splits, then $\sE \simeq \cO ^{\oplus 2}$ since the other contraction $p$ is  of fiber type by Condition~\ref{cond:*}.
Hence $X \simeq \bP^1 \times Y$.

If $\sE$ is the universal subbundle on $G(2,5)$, then $X \simeq F(1,2;5)$.

\medskip
Second assume $\dim Y = 5$.
Then by Condition~\ref{cond:*}, $p$ is a $\bP ^{2}$-bundle and hence the assertion follows from Theorem~\ref{thm:char}.

\medskip
Finally we assume that $\dim Y = 4$.
Then $\pi$ is a $\bP^3$-bundle or a $\bQ^3$-bundle on $Y$, and we have $\dim Z = 3$ or $4$.

If $\dim Z = 3$, then $-K_\pi$ is nef by \cite[Proposition~3.1]{Kan15b}, and hence $-K _\pi$ defines the contraction $p$.
Then, for a $p$-fiber $F'$, we have
\[
-K_{F'} = -K_X|_{F'}=( -\pi^*K_Y  -K_\pi)|_{F'} =-\pi^*K_Y|_{F'}.
\]
Therefore the morphism $F '\to Y$ is \'etale, and hence isomorphism.
This implies that $X \simeq Y \times Z$, where $Y$ and $Z$ are homogeneous by Condition~\ref{cond:*}.

Hence we may assume that $\dim Z = 4$.
In this case it is enough to show that $\pi$ and $p$ are smooth $\bP^3$-fibrations by \cite[Theorem~2]{OW02}.
Assume to the contrary that one of the morphisms is a smooth $\bQ^3$-fibration.
We may assume that $\pi$ is a smooth $\bQ^3$-fibration.
Then, by the Serre spectral sequence, $H^2(X,\bZ) \to H^2(F,\bZ)$ is surjective, where $F$ is a $\pi$-fiber.
Hence there exists a vector bundle $\sE$ of rank five on $Y$ such that $X$ is a (relative) quadric in $\bP(\sE)$, more precisely; 
\begin{enumerate}
 \item $X \subset \bP(\sE)$ and $X \in |2 \,\xi + m \,\varphi ^{*} H_{Y}|$, where $\xi$ is the tautological divisor on $\bP(\sE)$ and $H_{Y}$ is the ample generator of $\Pic (Y)$.
 \item $\sE ^{*} \simeq \sE (m H_{Y})$ by the section $s \in H ^{0}(S^{2}\sE (mH_{Y}))$ corresponding  to $X \in |2 \,\xi + m \,\varphi ^{*} H_{Y}|$.
\end{enumerate}
\[\xymatrix{
    X   \ar[d]_{\pi}  \ar@{^{(}->}[r]^{i}    &   \bP(\sE) \ar[ld]^{\varphi} \\
Y   .       &                     \\
}\]
Since the rank of $\sE$ is odd, $m$ is an even number by (2).
Hence we may assume that $m=0$ by twisting $\sE$ with a line bundle. 

Note that, since $\sE \simeq \sE ^*$, odd Chern classes vanish and the Grothendieck relation of $\bP(\sE)$ is 
\begin{align}\label{eq:Gro}
\xi ^5 + \varphi^* c_2(\sE) \cdot \xi^3 + \varphi^* c_4(\sE) \cdot \xi =0.
\end{align}

Also note that $-K_\pi = 3 \xi|_X$ by the adjunction.

Let $\tau$ be the slope for $\pi \colon X \to Y$, that is $-K_\pi + \tau \pi^*H_Y $ is nef but not ample
(cf.\ Subsection~\ref{subsec:slope} for projectivised vector bundles).
Then $g$ is defined by the divisor $-K_\pi + \tau \pi^*H_Y $, and hence
\[
( -K_\pi + \tau \pi^*H_Y)^6\cdot \pi^*H_Y=(-K_\pi + \tau \pi^*H_Y) ^5 \cdot \pi^*H_Y^2 =0.
\]
Since $X \in |2\xi|$, we rewrite these as follows:
\[
( 3\xi + \tau \varphi^*H_Y)^6\cdot \varphi^*H_Y \cdot \xi=(3 \xi + \tau \varphi^*H_Y) ^5 \cdot \varphi^*H_Y^2 \cdot \xi =0
\text{ on $\bP(\sE)$.}
\]
Combining with the Grothendieck relation \eqref{eq:Gro}, we have 
\begin{align*}
 10 H_Y^4\, \tau^2 -3^2 c_2(\sE) H_Y^2=0, \\
 10 H_Y^4\, \tau^3 -3^3 c_2(\sE) H_Y^2 \, \tau=0.
\end{align*}
This implies that $\tau = 0$ and $c_2(\sE) H_Y^2 =0$.
Hence $-K_\pi$ is nef but not ample, and defines the other contraction $p$.
It follows from the Grothendieck relation \eqref{eq:Gro} and $c_2(\sE) H_Y^2 =0$ that $(-K_\pi)^4 \cdot \pi^*H_Y^3=0$.
Hence we have $\dim Z = 3$, which contradicts our assumption $\dim Z =4$.
\end{proof}

Second we show the assertion in the case of $8$-folds with Picard number three:
\begin{proposition}\label{prop:8-folds}
Let $X$ be a Fano $8$-fold with Picard number three which satisfies Condition~\ref{cond:*}.
Then $X$ is a complete flag manifold or  a product of two Fano manifolds with Condition~\ref{cond:*}.
In particular, Theorem~\ref{thm:rhon-6} holds in this case.
\end{proposition}

\begin{proof}
By the assumption on $X$, there exists the following commutative diagram whose arrows are pairwise distinct elementary rational homogeneous fibrations by Proposition~\ref{prop:cont}:
\begin{align}\label{diagram}
\begin{gathered}
\xymatrix@=10pt{
X  \ar[dd]_{f_2} \ar[rd]^{f_1} \ar[rr] && X_3 \ar'[d][dd] \ar[rd] & \\
& X_1 \ar[dd]^(0.4){g_1} \ar[rr] && X_{3,1} \ar[dd] \\
X_2 \ar'[r][rr] \ar[rd]_{g_2} && X_{2,3} \ar[rd] \\
& X_{1,2} \ar[rr] && \text{pt}.
}
\end{gathered}
\end{align}
Set $h \coloneqq g_1 \circ f_1 = g_2 \circ f_2 $.

If every elementary contraction of $X$ is a smooth $\bP^1$-fibrations, then $X$ is a complete flag manifold by \cite{OSWW14}, and the assertion follows.
If $X$ admits a contraction $\pi$ onto a complete flag manifold $M$, then $X$ is isomorphic to the product of a $\pi$-fiber and $Y$ by Proposition~\ref{prop:iniFT}, and the assertion follows also in this case.

\medskip
Therefore, we may assume that $X$ has an elementary contraction which is not a smooth $\bP^1$-fibration and that $X$ does not admit a contraction onto a complete flag manifold.
By renumbering if necessary, we may assume that $\dim X_1 \leq 6$.
Then, by our assumption and Propositions~\ref{prop:cont} and \ref{prop:rhm},  $X_1$ is one of the following homogeneous manifolds:
\begin{enumerate}
 \item \label{X11}$\bP^2 \times \bP^4$, $\bP^2 \times \bQ^4$,  $\bP^2 \times \bP^3$, $\bP^2 \times \bQ^3$, $(\bP^2)^2$,
 \item \label{X12}$\bP(\sS _i)$, $\bP(T_{\bP^3})$,
  \item \label{X13}$\bP^3 \times \bQ^3$, $(\bP^3)^2$, $(\bQ^3)^2$.
\end{enumerate}

If Case~\ref{X11} or \ref{X12} occurs, then we may assume that $g_1$ is a $\bP^2$-bundle.
Every $h$-fiber is a Fano manifold with Condition~\ref{cond:*} of dimension $4$, $5$ or $6$ by Proposition~\ref{prop:cont} and it admits a contraction onto $\bP^2$.
Then every $h$-fiber is isomorphic to $\bP^2 \times \text{(an $f_1$-fiber)}$ by Proposition~\ref{prop:rhm}.
Hence the square on the left of \eqref{diagram} is a Cartesian product and $f_2$ is a $\bP^2$-bundle.

Therefore, if Case~\ref{X11} occurs, then we have $X \simeq X_{3,1} \times X_2$ since the squares on the left and the front are Cartesian products.

On the other hand, if Case~\ref{X12} occurs, then we have $X_{1,2} \simeq \bP^3$.
Then $\text{(an $f_1$-fiber)} \simeq \text{(a $g_2$-fiber)} \simeq \bP^3$ or $\bQ^3$.
Hence  $X_2 \simeq X_{1,2} \times \text{(a $g_2$-fiber)}$ and $X_{2,3} \simeq \text{(a $g_2$-fiber)}$ by Propositions~\ref{prop:cont} and \ref{prop:rhm}.
Then the square on the bottom of \eqref{diagram} is a Cartesian product and $X \simeq X_1 \times \bP^3$ or $X_1 \times \bQ^3$. 

\medskip
Assume that Case~\ref{X13} occurs.
If the square on the left of \eqref{diagram} is a Cartesian product, then $X \simeq X_2 \times X_{3,1}$ and the assertion follows.
Hence we may assume that the square on the left is not a Cartesian product.
Then $h$-fibers are isomorphic to $\bP(\sS _i)$ or $\bP(T_{\bP^3})$ by Propositions~\ref{prop:cont} and \ref{prop:rhm}.
In particular we have $g_1$ is a $\bP^3$-bundle.
In this case, $f_2 \colon X \to X_2$ is the family of linear subspaces in $g_1$-fibers.
By the universality of Hilbert schemes, we have $X \simeq X_{1,2} \times  \bP(\sS _i)$ or $X_{1,2} \times \bP(T_{\bP^3})$.
\end{proof}

We prove the following, which completes the proof of our main theorem:

\begin{theorem}
Let $X$ be a Fano $n$-fold with  Condition~\ref{cond:*} and Picard number $\rho_X = n-5$.
Then $X$ is 
a rational homogenous manifold or 
$(\bP^1)^{n-7} \times ({X_0} \text{ as in Definition~\ref{def:Ottaviani}})$.
\end{theorem}

\begin{proof}
We proceed by induction on $n$.
The assertion in the case of $n=6$ follows from the definition and the assertion in the cases $n=7$ or $n=8$ follows from Propositions~\ref{prop:7-folds} and \ref{prop:8-folds}.
If $n>8$, then $n\leq 2 \rho_X +1$ holds.
Therefore  $X$ is homogeneous or a product $Y \times M$ with $M$ is a complete flag manifold by Proposition~\ref{prop:2rho}.
In the latter case, since 
\[
\dim Y - \rho_Y \leq \dim X - \rho_X =5,
\]
$Y$ is a rational homogenous manifold or $(\bP^1)^{\dim Y-7} \times {X_0}$ by our inductive hypothesis.

If $Y$ is a rational homogenous manifold, then the assertion follows.
If $Y$ is isomorphic to $(\bP^1)^{\dim Y-7} \times {X_0}$, then $\dim M = \rho_{M} $.
Hence $M \simeq (\bP^1)^{\dim X - \dim Y}$  by \cite[Proposition 2.4]{BCDD03}, \cite[Proposition 5.1]{NO07} or \cite[Proposition~2.3]{Wat14a}.
This completes the proof.
\end{proof}

Finally, we prove Corollary~\ref{cor:CPn-6}.

\begin{proof}[Proof of Corollary~\ref{cor:CPn-6}]
By Theorem~\ref{thm:rhon-6},
it is enough to show that every CP $n$-fold with $\rho _X > n-6$ satisfies Condition~\ref{cond:*}.
The proof is proceeded by induction on $\rho _X$.
Note that every CP manifold with Picard number one and dimension at most five is a rational homogeneous manifold by \cite{CP91,Hwa06,Mok02,Kan15a}.
Hence, by our assumption, every CP manifold with Picard number one and dimension at most six is a rational homogeneous manifold, and hence satisfies Condition~\ref{cond:*}.

Let $X$ be a CP $n$-fold with $\rho _X > n-6$ and $\rho_X >1 $.
Suppose that 
\[
X  \xrightarrow{f_1}  X_1  \xrightarrow{f_2} \cdots \xrightarrow{f_{m-1}} X_{m-1} \xrightarrow{f_m}  X_m
\]
is a sequence of elementary contractions.
Then $X_1$ is again a CP manifold by \cite[Theorem~5.2]{DPS94}, \cite[Theorem~4.4]{SW04} or \cite[Proposition~4]{MOSW15} (cf.\ Proposition~\ref{prop:cont}).
Hence by our inductive hypothesis $X_1$ satisfies Condition~\ref{cond:*}, and hence $f_i$ for $i \geq 2$ are rational homogeneous fibrations.
On the other hand, by \cite[Proposition~4]{MOSW15}, every $f_1$-fiber $F$ is a CP manifold with Picard number one and  $\rho_{X_1} \leq \dim X_1$.
Then we have
\[
\dim F = \dim X - \dim X_1 \leq \dim X - \rho_{X_1} = \dim X - \rho_X +1 <7.
\]
Hence $F$ is a rational homogeneous fibration by our assumption and the assertion follows.
\end{proof}

\bibliographystyle{amsplain}
\bibliography{../references20150202}

\end{document}